\documentclass[a4paper,12pt]{article}

\usepackage{graphicx}
\usepackage{amsmath,amsthm,amsfonts,amssymb,amscd}
\usepackage[active]{srcltx}

\newtheorem{theo}{Theorem}[section]
\newtheorem{prop}{Proposition}[section]
\newtheorem{lemma}{Lemma}[section]
\newtheorem{cor}{Corollary}[section]

\begin{document}

\title{Center foliation: absolute continuity, disintegration and rigidity}
\author{R\'egis Var\~ao
\thanks {regisvarao@icmc.usp.br}}



\maketitle

 \begin{abstract}
In this paper we address the issues of absolute continuity for the center foliation (as well as the disintegration on the non-absolute continuous case) and rigidity  of volume preserving partially hyperbolic diffeomorphisms isotopic to a linear Anosov on $\mathbb T^3$. It is shown that the disintegration of volume on center leaves may be neither atomic nor Lebesgue, in contrast to the dichotomy (Lebesgue or atomic) obtained by Avila, Viana, Wilkinson \cite{AVW}. It is also obtained results concerning the atomic disintegration. Moreover, the absolute continuity of the center foliation does not imply smooth conjugacy with its linearization. Imposing stronger conditions besides absolute continuity on the center foliation, smooth conjugacy is obtained.
\end{abstract}

\section{Introduction and statements}
We study the measure-theoretical properties of the center foliation of partially hyperbolic diffeomorphisms for which the center leaves are non-compact. Two main issues are:

\begin{itemize}
 \item absolute continuity: when is the center foliation absolutely continuous? What can be said otherwise (disintegration)?
\item rigidity: does absolute continuity imply greater regularity?
\end{itemize}

These issues are fairly well understood for certain volume preserving partially hyperbolic diffeomorphism (perturbations of certain skew-products or of time-one maps of Anosov flows) studied by Avila, Viana and Wilkinson \cite{AVW}. We state their dichotomy:

\begin{itemize}
\item \textit{Atomic disintegration}: If the center foliation is non-absolutely continuous, then there exists $k \in \mathbb N$ and a full volume subset that intersects each center leaf on exactly $k$ points/orbits.
 \item \textit{Rigidity}: If the center foliation is absolutely continuous then the diffeomorphism is smoothly conjugate to a rigid model (a rotation extension of an Anosov diffeomorphism or the time-one map of an Anosov flow);
\end{itemize}

On three-dimensional manifolds, the only known examples of partially hyperbolic diffeomorphism are of skew-product type, perturbation of time-one of Anosov flows and diffeomorphisms derives from a linear Anosov. In fact it is conjecture by E. Pujals that these are all the possibilities (see \cite{bonatti.wilkinson} for precise statements). Avila, Viana and Wilkinson \cite{AVW} have treated diffeomorphisms on the first two cases and we treat in this work the third case.

We deal with derived from Anosov (DA) diffeomorphisms, that is, $f$ is a DA diffeomorphism if it partially hyperbolic and lies in the isotopy class of some hyperbolic linear automorphism $A$. We refer to $A$ as the the linearization of $f$. Every partially hyperbolic diffeomorphism has the following splitting on the tangent space $TM=E^s\oplus E^c\oplus E^u$ (see \S \ref{sec.preliminaries} for definitions), where $E^s$ is a contracting direction, $E^u$ is an expanding direction and the center direction $E^c$ has an intermediate behavior. We consider Anosov diffeomorphisms (see \S \ref{sec.preliminaries} for definition) which are partially hyperbolic. That means that they have the splitting $TM=E^s\oplus E^c\oplus E^u$ and the center direction is uniformly contracting or expanding. It then makes sense to talk about the center foliation of an Anosov (partially hyperbolic) diffeomorphism, as we do on Theorem \ref{theo.not.leb.not.atomic} and Theorem \ref{theo.abs.cont.not.c1}. A center foliation (see \S \ref{sec.preliminaries}) is an invariant foliation by $f$ tangent to the $E^c$ direction. We also mention that all diffeomorphisms treated on this work are assumed to be at least $C^{1+\alpha}$. This implies, in particular, that volume preserving Anosov on $\mathbb T^3$ are ergodic.

We state now our results. For non-absolutely continuous center foliation,  we show that it is possible to have a disintegration (\S \ref{sec.decomposition} for definition) which is non-Lebesgue and non-atomic. Our conclusion is different from \cite{AVW} and indeed this is the first example of this kind:

\begin{theo}\label{theo.not.leb.not.atomic} For partially hyperbolic Anosov diffeomorphisms on $\mathbb T^3$, volume preserving, for which the center foliation is non-absolutely continuous, we have that

\begin{itemize}
 \item[i)] there exists $f$ Anosov for which the disintegration of volume on the center leaves are neither Lebesgue, nor atomic.

In fact, such diffeomorphisms fill a dense subset of an infinite-dimensional manifold in the neighborhood of any hyperbolic linear automorphisms in the space of volume preserving maps;

 \item[ii)] the conditional measures are singular measures with respect to the volume on the center leaf;

\item[iii)] if the decomposition is atomic, then there is exactly one atom per leaf. That is, there exists a set of full volume that intersects each center leaf in one point;

\item[iv)] the disintegration of volume on the center leaves is atomic if and only if the partition by center leaves is a measurable partition.
\end{itemize}
\end{theo}

For item iv) above, we don't need to suppose that we are on the non-absolute continuous case. The next result shows, in contrast to the dichotomy \cite{AVW}, that absolute continuity has no rigidity implications in our case:

\begin{theo}\label{theo.abs.cont.not.c1}
 There exist volume preserving Anosov diffeomorphisms $f$ on $\mathbb T^3$ for which the center foliation is absolutely continuous but $f$ is not $C^1$-conjugate to its linearization.

In fact, such diffeomorphisms fill a dense subset of an infinite-dimensional manifold in the neighborhood of any hyperbolic linear automorphism in the space of volume preserving maps.
\end{theo}

As we shall see, Theorem \ref{theo.abs.cont.not.c1} will be just a corollary of the following result, which is important on its own:

\begin{lemma}\label{lemma.conjugacy}
  Let $f$ be a volume preserving partially hyperbolic Anosov diffeomorphism on $\mathbb T^3$. Then, for any periodic points $p, q$ the Lyapunov exponents on each of the directions (stable, center, unstable) are the same if and only if $f$ is $C^1$ conjugate to its linearization.
\end{lemma}

Note that Theorem \ref{theo.abs.cont.not.c1} implies that to obtain a rigidity result we must impose some stronger conditions on the center foliation besides absolute continuity. And we do so to obtain the following rigidity result.

\begin{theo}\label{theo.rigidity}
  Let $f$ be a volume preserving DA diffeomorphism on $\mathbb T^3$, with the linearization $A$. If the center foliation is a $C^1$ foliation and the center holonomies inside the center-unstable, $\mathcal F_f^{cu}$, and center-stable, $\mathcal F_f^{cs}$, leaves are uniformly bounded, then $f$ is $C^1$ conjugate to its linearization and, hence, is an Anosov diffeomorphism.
\end{theo}

\textit{Organization of the paper.} In \S \ref{sec.preliminaries} we give some basic definitions such as what we mean by to disintegrate a measure, absolute continuity, etc. In \S \ref{sec.non-abs.cont} we study the behavior of non-absolute continuous center foliation, where we prove Theorem \ref{theo.not.leb.not.atomic}. We begin \S \ref{sec.conjugacy} understanding how Lyapunov exponents vary with respect to their linearization, we then prove Lemma \ref{lemma.conjugacy} and Theorem \ref{theo.abs.cont.not.c1}. In \S \ref{sec.rigidity} we construct some conditional measures (not probabilities) on each center leaf with some dynamical meaning. We use these measures to prove Theorem \ref{theo.rigidity}.

\section{Preliminaries}\label{sec.preliminaries}

\subsection{Partially Hyperbolic Diffeomorphism.}

A diffeomorphism $f$ of a compact Riemannian manifold $M$ is called partially hyperbolic  if there are constants $\lambda < \hat{\gamma} < 1 < \gamma < \mu$ and $C > 1$ and a $Df$ -invariant splitting of $ T M = E^u (x) \oplus E^c (x) \oplus E^s (x)$ where
\begin{eqnarray*}
 \frac{1}{C}\mu^n||v|| < & ||Df^nv||, &   \quad \quad  \quad  \quad \quad  \; \; v \in E^u_x - \{0\}; \\
 \frac{1}{C}\hat{\gamma}^n||v|| < &||Df^nv||& < C \gamma^n ||v||, \; \;  v \in E^c_x - \{0\}; \\
		&||Df^nv||&< C \lambda^n ||v||,   \; \;  v \in E^s_x - \{0\}.
\end{eqnarray*}

 We say that a partially hyperbolic diffeomorphism is \textit{dynamically coherent} if the subbundles $E^s \oplus E^c$ and $E^c \oplus E^u$ integrate into invariant foliations, $\mathcal F^{cs}, \mathcal F^{cu}$ respectively. This implies in particular that there is a center foliation $\mathcal F^c$, which is obtained by an intersection of the other two: $\mathcal F^c = \mathcal F^{cs} \cap \mathcal F^{cu}$. It was proved by Brin, Buragov, Ivanov \cite{BBI-T^3} that 

\begin{theo}
Every partially hyperbolic diffeomorphism on $\mathbb T^3$ is dynamically coherent. 
\end{theo}

\subsection{Decomposition of measure}\label{sec.decomposition}

Let $(M, \mu, \mathcal B)$ be a probability space, where $M$ is a compact metric space, $\mu$ a probability and $\mathcal B$ the borelian $\sigma$-algebra.
Given a partition $\mathcal P$ of $M$ by measurable sets, we associate the following probability space $(\mathcal P, \widetilde \mu, \widetilde{\mathcal B})$, where $\widetilde \mu := \pi_* \mu$, $ \widetilde{\mathcal B}:= \pi_*\mathcal B$. and $\pi:M \rightarrow \mathcal P$ is the canonical projection associate to a point of $M$ the partition element that contains it.
 
For a given a partition $\mathcal P$, a family $\{\mu_P\}_{p \in \mathcal P}$ is a \textit{system of conditional measures} for $\mu$ (with respect to $\mathcal P$) if
\begin{itemize}
 \item[i)] given $\phi \in C^0(M)$, then $P \mapsto \int \phi \mu_P$ is measurable;
\item[ii)] $\mu_P(P)=1$ $\widetilde \mu$-a.e.;
\item[iii)] if $\phi \in C^0(M)$, then $\displaystyle{ \int_M \phi d\mu = \int_{\mathcal P}\left(\int_P \phi d\mu_P\right)d\widetilde \mu}$.
\end{itemize}

We call $\mathcal P$ a \textit{measurable partition} (w.r.t. $\mu$) if there exist a family $\{A_i\}_{i \in \mathbb N}$ of borelian sets and a set $F$ of full $\mu$-measure such that for every $P \in \mathcal{P}$ there exists a sequence $\{B_i\}_{i \in \mathbb N}$, where $B_i \in \{ A_i, A_i^c\}$ such that $P \cap F = \cap_{i \in \mathbb N}B_i \cap F$.

The following result is also known as Rokhlin's disintegration Theorem.

\begin{theo}\label{theo:rokhlin} 
 Let $\mathcal P$ be a measurable partition of a compact metric space $M$ and $\mu$ a borelian probability. Then there exists a disintegration by conditional measures for $\mu$.
\end{theo}

\textit{Remark.} On Theorem \ref{theo.not.leb.not.atomic} the meaning of ``disintegration of volume on the center leaves are neither Lebesgue, nor atomic'' means that on a foliated box, since the center foliation form a measurable partition we can apply on this foliated box the Rokhlin's disintegration Theorem and the conditional measures are neither Lebesgue, nor atomic. This is independent of the foliated box (see Lemma \ref{lemma:glue.measures.AVW}) and that is why we don't say instead that the disintegration locally is neither Lebesgue nor atomic.

\subsubsection{Absolute continuity}

Let $\mathcal F$ be a foliation and disintegrate the volume inside a foliated box. If the conditional measure $m_L$ on the leave satisfies that $m_L << Leb_L$ for almost every leaf, then $\mathcal F$ is said to be an \textit{absolutely continuous foliation}, where $Leb_L$ is the Lebesgue measure on the leaf $L$.

We state a result due to Gogolev \cite{gogolev-non.abs.cont} which shall be our starting point to understand absolute continuity for partially hyperbolic diffeomorphism with non-compact center leaves.

\begin{theo}\label{theo:gogolev}
Let $f:\mathbb T^3 \rightarrow \mathbb T^3$ be an Anosov diffeomorphism with splitting of the form $E^s\oplus E^{wu} \oplus E^{uu}$, then $\mathcal F^c_f$ is absolutely continuous if and only $\lambda^{uu}(p) = \lambda^{uu}(q)$ for all periodic points $p$ and $q$.
\end{theo}

Where $\lambda^{uu}$ is the Lyapunov exponent on the $E^{uu}$ direction.

\subsection{Geometric property}

By a Derived from Anosov (DA) diffeomorphism $f:\mathbb T^3 \rightarrow \mathbb T^3$ we mean a partially hyperbolic homotopic to a linear Anosov diffeomorphism $A$. We call this linear Anosov as the linearization of $f$. In fact, $f$ is semi-conjugated to its linearization. The itens from the Theorem below, which proof can be found on Sambarino \cite{sambarino.hiperbolicidad.estabilidad}, show that the semi-conjugacy has in fact good properties.

\begin{theo}
 Let $B :\mathbb R^3 \rightarrow \mathbb R^3$ be a linear hyperbolic isomorphism. Then, there exists $C>0$ such that if $G: \mathbb R^3 \rightarrow \mathbb R^3$ is a homeomorphism such that $sup\{ ||G(x) - Bx|| \; | \; x \in \mathbb R^3 \} = K < \infty$ then there exists $H:\mathbb R^3 \rightarrow \mathbb R^3$ continuous and surjective such that:
\begin{itemize}
 \item $B \circ H = H \circ G$;
\item $||H(x) - x|| \leq CK$ for all $x \in \mathbb R^3$;
\item $H(x)$ is characterized as the unique point $y$ such that
$$ ||B^n(y) - G^m(x)|| \leq CK, \; \forall n \in \mathbb Z;$$
\item $H(x) = H(y)$ if and only if $||G^n(x) - G(y)|| \leq 2CK$, $\forall n \in \mathbb Z$, and if and only if $sup_{n \in \mathbb Z}\{ || G^n(x) - G^n(y)|| \} < \infty$;
\item if $B \in SL(3,\mathbb Z)$ and $G$ is the lift of $g: \mathbb T^3 \rightarrow \mathbb T^3$ then $H$ induces $h:\mathbb T^3 \rightarrow \mathbb T^3$ continuous and onto such that $B \circ h = h \circ g$ and $dist_{C^0}(h, id) \leq C dist_{C^0}(B,g)$.
\end{itemize}

\end{theo}
The geometrical property we shall need later is given by Hammerlindl \cite{andy.hammerlindl-thesis}:

\begin{prop}\label{prop:geometric.property}
 Let $f$ be a partially hyperbolic and $A$ be its linearization. Denote by $\tilde f$ and $\tilde A$ the lift to $\mathbb R^n$ of $f$ and $A$ respectively. Then for each $k \in \mathbb Z$ and $C>1$ there is $M>0$ and a linear map $\pi:\mathbb R^n\rightarrow \mathbb R^n$  such that for all $x, y \in \mathbb R^n$
$$||x-y||> M \Rightarrow \frac{1}{C} < \frac{||\pi(\tilde f^k(x)-\tilde f^k(y))||}{||\pi(\tilde A^k(x) -\tilde A^k(y))||} < C.$$
\end{prop}

\section{Non-absolute continuity}\label{sec.non-abs.cont}

We dedicate this section for the proof of Theorem \ref{theo.not.leb.not.atomic}.

\subsection{Proof of item i)}

Consider a linear volume preserving Anosov with the following split $TM = E^{ss} \oplus E^{ws} \oplus E^{u}$. Let $\phi$ be a volume preserving diffeomorphism which preserves the $E^u$ direction. By Baraviera, Bonnatti \cite{baraviera.bonatti} $\int \lambda^{ws}_A \; dVol \neq \int \lambda^{ws}_{A\circ \phi} \; dVol$. Let $h$ be the conjugacy between $A$ and $f$, $f \circ h = h \circ A$. We claim that $h$ is volume preserving and sends center leaves to center leaves. To see that $h$ is volume preserving note that $f$ and $A$ have the same topological entropy $\lambda^{u}_A$. Hence, $h_*Vol$ is a measure of maximal entropy. Observe that the perturbation $A\circ \psi$ of $A$ is such that it preserves the $E^u$ exponent, which means that by the equilibrium state theory (see Bowen \cite{bowen-book}) the potentials $0$ and $-log ||Df_{|E^u}||$ are cohomological and therefore give the same equilibrium states. That is, $h_*Vol = Vol$. And the fact that $h(\mathcal F^c)= \mathcal F^c$ comes from Lemma 2 of \cite{
gogolev.guysinsky-c1.conjugacy}.

\textit{Claim:} $\mathcal F^c_{A\circ \psi}$ is not absolutely continuous.

Suppose, by contradiction, that it is absolutely continuous, then Theorem \ref{theo:gogolev} implies $\lambda^{ss}_f(p)= cte$ for all periodic point $p$. By contruction we have $\lambda^{u}_f(p)=\lambda^u_A$. Since we are on the volume preserving case, $\lambda^{ws}(p)$ is also constant on periodic points. Therefore, by Lemma \ref{lemma.conjugacy} $f$ is $C^1$-conjugate to $A$, but this would imply $\int \lambda^{ws}_f d Vol = \int \lambda^{ws}_A d Vol$. Which is absurd by Proposition 0.3 of Baravieira, Bonnatti \cite{baraviera.bonatti}.

\textit{Claim:} The disintegration of volume on center leaves of $A \circ \phi$ is neihter Lebesgue nor atomic.

It is not Lebesgue because it is not absolutely continuous. And to see that it is not atomic, note that since $h$ is volume preserving and sends center leaves onto center leaves we can induce (by push forward) the disintegration on the center leaves of $A$ to the center leaves of $A \circ \phi$. And since the disintegration for $A$ is Lebesgue, this means that the disintegration for $A \circ \phi$ is not atomic. \hfill $\Box$

\subsection{Proof of item ii)}

By ergodicity we know that the Birkhoff set $$ B=\{ x \in \mathbb T^3 \; | \; 1/n \sum_{i=0}^{n-1} \delta_{f^i(x)} \rightarrow Vol \text{ as } n \rightarrow \infty\} $$ has full measure.

\textit{Claim.}
 If there is a center leaf such that $\mathcal F^c \cap B$ has positive Lebesgue measure, then
the center foliation is absolutely continuous.

\textit{Proof of the Claim.} Let $D$ be any disc on the central foliation and consider the following construction
$$\mu_n = \frac{1}{n}\sum_{j=0}^{n-1} f_*^j\left(\frac{m_{D}}{m_D(D)}\right),$$
where $m_D$ means the Lebesgue measure on the central leaf. It turns out that these measures converge to a measure $\mu$ such that the disintegration of $\mu$ on the center leaves are absolutely continuous with respect to the Lebesgue measure. This is a well-known construction of measures, studied by Pesin, Sinai in the eighties. For more references see \cite{bonatti.diaz.viana} Chapter 11 and the references therein. Although Pesin, Sinai studied these measures for the case of the disc $D$ in the unstable foliation, for the center foliation, in our case, this construction is the same. Gogolev, Guysinsky \cite{gogolev.guysinsky-c1.conjugacy} have worked explicitly on this case and the reader may check at \cite{gogolev.guysinsky-c1.conjugacy} the construction.

We make a slightly different construction, instead of the disc $D$, as above, we take $D\cap B$ for which it has positive Lebesgue measure on the center leaf. By hypothesis there exists such a disc. It turns out that these measures still converge to a measure with conditional measures absolutely continuous to the Lebesgue measure on the center leaf (Lemma 11.12 \cite{bonatti.diaz.viana}). Since the points on $B$ have the property $1/n \sum_{i=0}^{n-1} \delta_{f^i(x)} \rightarrow Vol$, it turns out that the sequence $\mu_n$ converges to the volume. Hence, volume has Lebesgue disintegration on the center leaves. Which proves the claim.

From the claim, since we are in the case where the center foliation is non-absolutely
continuous, we must have that the center foliation intersects $B$ on a set of zero Lebesgue
measure. But the conditional measures give full measure to $B$, since $B$ has full measure. Therefore the conditional
measures are singular with respect to the Lebesgue measure. And item ii) is proved. \hfill $\Box$

\subsection{Proof of item iii)} 
On what follows $R_i$ will denote a rectangle of a fixed finite Markov partition. The proof of item iii) will be a consequence of the following lemmas.

\begin{lemma}
 All the atoms have the same weight when considering the disintegration of volume on the center leaves of $R_i$.
\end{lemma}
\begin{proof} On each Markov rectangle we may apply Rokhlin's disintegration theorem on center leaves. Therefore, when writing $m_x$ we mean the conditional measure for the disintegration on Markov rectangle that contains $x$. Consider the set $A_\delta =\{ x \in A \; | \; m_{x}(x) \leq \delta \}$. Since $f(\mathcal F^{c}_{R(x)}(x)) \supset  \mathcal
F^{c}_{R(f(x))}(f(x))$, we
have that $f_*m_x(I) \leq m_{f(x)}(I)$ where $I$ is inside the connected component of $\mathcal F^c_{f(x)} \cap R(f(x))$ that contains $f^n(x)$. If $f(x) \in A_\delta$, then
$$m_x(x) = f_*m_x(f(x)) \leq  m_x(f(x)) \leq \delta.$$
Hence, $f^{-1}(A_\delta) \subset A_\delta$.

By ergodicity, since our Anosov is volume preserving on $\mathbb T^3$, $A_\delta$ has full measure or zero measure. Let $\delta_0$ be the discontinuity point of the function $\delta \in [0,1] \mapsto Vol(A_\delta)$. This implies that almost every atom has weight $\delta_0$.
\end{proof}

\begin{lemma}
 On every Markov partition $R_i$ the conditional measures have the same number of atoms,
with the same weight.
\end{lemma}

\begin{proof} This is a direct consequence from the above lemma. Since all the atoms have the same
weight $\delta_0$ the conditional measures must have $1/\delta_0$ number of atoms.
\end{proof}

\begin{lemma} 
There is a set of full volume $B_1$, of atoms, such that if $x \in B_1$, then $B_1 \cap \mathcal F^c_x$ is contained in the connected component of  $R_{i(x)} \cap \mathcal F^c_x$ that contains $x$.
\end{lemma}
\begin{proof} Let $A$ be the set of atoms and $T$ be the set of transitive points. Both sets have full
volume measure by ergodicity. Suppose, by contradiction, that there is a subset $A_1 \subset
A$ of positive volume measure such that $\forall x \in A_1$ we get  $A \cap R_{i(x)} ^c \neq
\emptyset$, where $R_{i(x)} ^c$ is the complement of the Markov partition that contains $x$, note that $Vol(A_1 \cap T) >0$. Define the following map

\begin{eqnarray*}
 h:  A_1 \cap T & \rightarrow & \mathbb R \\ 
     x  & \mapsto & h(x) = d_{\mathcal F^c _x}(R_{i(x)},R_{i(x)}'),
\end{eqnarray*}
where $d_{\mathcal F^c _x}(R_{i(x)},R_{i(x)}')$ means the distance inside the center leaf of
the Markov rectangle $R_{i(x)}$ to the closest Markov
rectangle that has an atom which we call $R_{i(x)}'$.

Since $h$ is a measurable map, there exists $K_1 \subset A_1 \cap T$, with $Vol(K_1)>0$ for which $h$ is a continuous map when restricted to $K_1$. And since volume is a regular measure, there is compact set $K_2 \subset K_1$, also with positive volume measure. 

Let $\alpha = Max_{x \in K_2} h(x)$. Fix $z_0 \in R_{i(z_0)}$, and consider a ball small enough such that $B(z_0, r) \subset int R_{i(z_0)}$. Hence, $\forall y \in K_2$, let $n_y \in \mathbb N$ be an integer big enough so that, since $f$ is uniformly expanding in the center direction, $f^{-n_y}(\mathcal F^c(y,\alpha)) \subset B(z_0,r) \subset int R_{i(z_0)}$.

It means that we have at least doubled the number of atoms inside $R_{i(z_0)}$, which is an absurd since we have already shown that the number of atoms are constant on each Markov partition.
\end{proof}
\begin{lemma} 
There is a set of full volume $B_1$, of atoms, such that if $x \in B_1$, then $B_1 \cap \mathcal F^c_x$ is contained in the connected component of  $R_{i(x)} \cap \mathcal F^c_x$ that contains $x$.
\end{lemma}
\begin{proof} Let $A$ be the set of atoms and $T$ be the set of transitive points. Both sets have full
volume measure by ergodicity. Suppose, by contradiction, that there is a subset $A_1 \subset
A$ of positive volume measure such that $\forall x \in A_1$ we get  $A \cap R_{i(x)} ^c \neq
\emptyset$, where $R_{i(x)} ^c$ is the complement of the Markov partition that contains $x$, note that $Vol(A_1 \cap T) >0$. Define the following map

\begin{eqnarray*}
 h:  A_1 \cap T & \rightarrow & \mathbb R \\ 
     x  & \mapsto & h(x) = d_{\mathcal F^c _x}(R_{i(x)},R_{i(x)}'),
\end{eqnarray*}
where $d_{\mathcal F^c _x}(R_{i(x)},R_{i(x)}')$ means the distance inside the center leaf of
the Markov rectangle $R_{i(x)}$ to the closest Markov
rectangle that has an atom which we call $R_{i(x)}'$.

Since $h$ is a measurable map, there exists $K_1 \subset A_1 \cap T$, with $Vol(K_1)>0$ for which $h$ is a continuous map when restricted to $K_1$. And since volume is a regular measure, there is compact set $K_2 \subset K_1$, also with positive volume measure. 

Let $\alpha = Max_{x \in K_2} h(x)$. Fix $z_0 \in R_{i(z_0)}$, and consider a ball small enough such that $B(z_0, r) \subset int R_{i(z_0)}$. Hence, $\forall y \in K_2$, let $n_y \in \mathbb N$ be an integer big enough so that, since $f$ is uniformly expanding in the center direction, $f^{-n_y}(\mathcal F^c(y,\alpha)) \subset B(z_0,r) \subset int R_{i(z_0)}$.

It means that we have at least doubled the number of atoms inside $R_{i(z_0)}$, which is an absurd since we have already shown that the number of atoms are constant on each Markov partition.
\end{proof}

\begin{lemma}
There is a set of full volume $B_2 \subset B_1$ such that the center foliation intersects $B_2$ at most
on one point.
\end{lemma}
\begin{proof} By contradiction suppose that the number of atoms on all Markov partition are greater than one. Let $A_2$ be a set with full volume measure inside the union of the Markov rectangle such that if $x \in A_2$, then $A_2 \cap \mathcal F^c _{x, loc}$ has the same number of points, in this case greater than one. Where $F^c _{x, loc}$ is the connected set of the center foliation restricted to the Markov rectangle that intersects $x$.  We define the map
\begin{eqnarray*}
 h : A_2 &\rightarrow &\mathbb R \\
 x &\mapsto & h(x)
\end{eqnarray*}
where $h(x)$ is the smallest distance between the atoms of $\mathcal F^c _{x, loc}$. By Lusin's theorem there is a set $K_1 \subset A_2$ of positive measure for which $h$ is continuous. Since volume is regular, there is a compact subset $K_2$ of $K_1$ with positive measure. Let $ \displaystyle \alpha = \min_{x \in K_2}h(x)$. 

Let $\beta >0$ be an inferior bound for the length of $\mathcal F^c_{loc}$. Let $n_0 \in \mathbb N$ big enough so that any segment of a center leaf with length greater than or equal to $\alpha$ has the length of its $n_0$th iterate greater than $\beta$. This means that $f^{n_0}(K_2)$, which has positive measure, have all the atoms separated from each other with respect to the Markov partition. Since we have a finite number of Markov partition, one of them must have a set with positive measure such that its leaves have only one atom. Hence all Markov partition must have one atom, absurd.
\end{proof}

\subsection{Proof of item iv)}

 Suppose $\{ \mathcal F^c_x\}_{x \in M}$ is a measurable partition, then we can apply Rokhlin's theorem and we decompose the volume on probabilities $m_x$ on center leaves. Let $$A_L =\{ x \in M \; | \; m_x(\mathcal F^c_L(x))\geq 0.6\},$$ where $ \mathcal F^c_L(x)$ is the segment of $\mathcal F^c(x)$ of length $L$ on the induced metric and centered at $x$.

Note that there is $L \in \mathbb R$ such that $vol(A_L)>0$. Let us suppose that $f$ contracts the center leaf, then $f^{-1}(\mathcal F^c_L(f(x))) \supset \mathcal F^c_L(x)$. Since $f_*m_x = m_{f(x)}$, for $x \in A_L$,
$$m_{f(x)}(\mathcal F^c_L(f(x))) = m_{x}(f^{-1}(\mathcal F^c_L(f(x)))) \geq m_x(\mathcal F^c_L(x)) \geq 0.6.$$

So $f(x) \in A_L$, by ergodicity $f(A_L)\subset A_L$ implies $Vol(A_L)=1$.

\textit{Claim:} $diam^c A_L \cap \mathcal F^c_x \leq 2L$, where $diam^c$ means the diameter of the set inside the center leaf.

Suppose there exist $y_1, y_2 \in A_L \cap \mathcal F^c_x$ with $d^c(y_1,y_2) >2L$. Then
$$\mathcal F^c_L(y_1) \cap \mathcal F^c_L(y_2) = \emptyset \; \text{ and } \; m_x(\mathcal F^c_L(y_i))\geq 0.6, \; i=1,2.$$
Then $$1 \geq  m_x(\mathcal F^c_L(y_1) \cup \mathcal F^c_L(y_2)) = m_x(\mathcal F^c_L(y_1)) + m_x(\mathcal F^c_L(y_2)) \geq 0.6 +0.6 = 1.2.$$
This absurd concludes the proof of the claim.

\textit{Claim:} The decomposition has atom.

Define $$L_0 = inf \{ L \in [0,\infty) \;|\; Vol(A_L)=1 \}.$$
Note that $Vol(A_{L_0})=1$, to see that take a sequence $L_n \rightarrow L_0$ and observe that $A_{L_0}= \cap_i A_{L_n}$. Let $\lambda = inf || Df^{-1}|E^c||$, let $\varepsilon < 1$ be such that $\varepsilon \lambda > 1$. For $x \in A_{\lambda L_0}$
$$m_{f(x)}(\mathcal F^c_{\varepsilon L_0}(f(x))) = m_{x}(f^{-1}(\mathcal F^c_{\varepsilon L_0}(f(x)))) \geq m_x(\mathcal F^c_{L_0}(x)) \geq 0.6.$$

Therefore $f(x)\in A_{\varepsilon L_0}$. By ergodicity we may suppose $A_{L_0}$ $f$-invariant, hence $Vol(A_{\varepsilon L_0})=1$. Absurd since $\varepsilon L_0 < L_0$. This means that $L_0 =0$, which implies atom. 

Let us prove the converse. Suppose we have atomic decomposition, we want to see that the partition through center leaves is a measurable partition.

Lift $f$ to $\mathbb R^3$, by Hammerlindl \cite{andy.hammerlindl-thesis} we may find a disk $\tilde D^2$ transverse to the center foliation, by quasi-isometry of the center foliation we may take this disk as big as we want. So take a disk such that its projection $D^2 = \pi(\tilde D^2)$ has the property:
$$\mathcal F^c_x \cap D^2 \neq \emptyset, \; \forall x \in \mathbb T^3.$$

Since the decomposition is atomic, we already know that it has one atom per leaf. Let us define the following set of full measure:
$$\hat M = \bigcup_{p \in A} \mathcal F^c_{loc}(p),$$
where $A$ is the set of atoms, $\mathcal F^c_{loc}(p)$ is the segment of center leaf such that the right extreme point is $p$ and the left extreme point is on $D^2$ and $\# \mathcal F^c_{loc}(p) \cap D^2 =1$.

Since $D^2$ is a separable metric space, $\{\mathcal F^c_{loc}(p)\}_{p \in A}$ is a measurable partition for $\hat M$. Therefore we have a family of subsets $\{A_i\}_{i \in \mathbb N}$ of $\hat M$ such for all $p \in A$
 $$\mathcal F^c_{loc}(p)= \bigcap_{i\in \mathbb N} B_i, \text{ where } B_i \in \{A_i, A_i ^c\}.$$
\hfill $\Box$

\section{Conjugacy}\label{sec.conjugacy}
We begin by understanding how Lyapunov exponents vary with respect to their linearization.

\begin{prop}\label{prop:lambdaf<lambdaA}
 Let $f:\mathbb T^3 \rightarrow \mathbb T^3$ be a partially hyperbolic, not necessarily ergodic nor volume preserving, and let $A$ be
its linearization. Then $\int \lambda^{u}(f) dVol \leq \lambda^{u}_A$.
\end{prop}
\begin{proof} Suppose that $\int \lambda^u_f(x) dVol(x) > \lambda^u_{A}$, then there exists a set $B$ of positive volume and a constant $\alpha$ such that $\lambda^u_f(x) > \alpha > \lambda^u_{f^*}$, $\forall x \in B$. Define 
$$B_N = \{ x \in B \; | \; ||Df^n|E^u_x|| \geq e^{n\alpha}; \; \forall n \geq N  \}.$$
Note that $$B=\bigcup_{N=1}^\infty B_N,$$ 
this means that there is $N_0$ such that $Vol(B_{N_0})>0$. Since $\mathcal F^u_f$ is absolutely continuous then there is $x \in B$ such that $\mathcal F^u_f(x)\cap B_{N_0}$ has positive volume on the unstable leaf.

Let $I \subset \mathcal F^u_f(x)$ be a compact segment with $Vol^c(I \cap B_{N_0})>0$ and $length(I)=:l(I) >M$. Then
\begin{eqnarray*}
 l(f^n(I)) &=& \int _{f^n(I)} d Vol^u = \int_I (f^n)^* dVol^u \geq \int_{I \cap A_{N_0}}(f^n)^* dVol^u \\
&\geq & \int_{I \cap B_{N_0}} ||Df^n|E^u_x|| d Vol^u(x) \geq e^{n\alpha} Vol^c(I \cap A_{N_0}).
\end{eqnarray*}

Consider $x,y$ the extremes of $I=[x,y]$. Then $d^u(f^n(x),f^n(y)) = l(f^n(I))$. Using quasi-isometry on the first inequality below we get

\begin{eqnarray*}
 \frac{d(f^n(x),f^n(y))}{d(A^n(x),A^n(y))} &\geq& cte \frac{d^u(f^n(x),f^n(y))}{d(A^n(x),A^n(y))} \\
 &\geq & cte \frac{e^{n\alpha}}{e^{n\lambda^u_A}} \frac{Vol(I \cap B_{N_0})}{d(x,y)}\\ &\longrightarrow& \infty \text{ as } n \rightarrow \infty.
\end{eqnarray*}

By Proposition \ref{prop:geometric.property} this ratio should be bounded. Absurd.
\end{proof}

The same type of argument above give us:
\begin{cor}\label{cor:expoente.estavel.menor}
 $$ \int \lambda^s(f) \geq \lambda^s(A).$$
\end{cor}
We consider the following for the case of Anosov systems, for it will be used later.
\begin{cor}\label{cor:lambdaf<lambdaA}
 Let $f$ be an Anosov diffeomorphism with the following split on the tangent space $TM = E^{ss}\oplus E^{ws}\oplus E^u$ and $\mathcal F^{ws}$ absolutely continuous. Then $\lambda^{ws}_f \geq \lambda^{ws}_A$.
\end{cor}
\begin{proof} The prove goes as before, with a minor change. We proceed, as previously, applying Proposition \ref{prop:geometric.property}  with the following linear map $\pi: \mathbb R^n \rightarrow \mathbb R^n$ which is the projection onto a center foliation of the linearization. The projection is with respect to the system of coordinate given by the foliations of the linearization $(x_{ss},x_{ws}, x_{u}) \in \mathbb R^n$.
\end{proof}

\subsection{Proof of Lemma \ref{lemma.conjugacy}}

We only have to prove the implication, as the converse is a direct consequence of the $C^1$-conjugacy.

 Let us suppose that $f$ is partially hyperbolic with the following split of the tangent space: $TM = E^{ss} \oplus E^{ws} \oplus E^{u}$. 
The next three lemmas concern this case, the other case is reduced to this one by applying the inverse.

\begin{lemma}
 $$\lambda^{u}_f(m) = \lambda^{u}_f(p), \; \forall p \in Per(f).$$
\end{lemma}
\begin{proof} By ergodicity the set of transitive points $\mathcal T$ has total volume. We may assume that all points of $\mathcal T$ have well defined Lyapunov exponents. For $x \in \mathcal T$; given $\varepsilon > 0$ let $\delta >0$ be such that by uniform continuity
$$|\; log||Df|E^u_{y_1}|| - log||Df|E^u_{y_2}|| \;| < \varepsilon, \text{ if } d(y_1,y_2)< \delta.$$

From the Shadowing lemma there is $\alpha$ such that for every $\alpha$-pseudo orbit is $\delta$ shadowed by a real orbit. Given $N_0 \in \mathbb N$ there is $n_0 \in \mathbb N$ and $n_0 > N_0$ such that $\{ \ldots, f^{n_0-1}(x), x, f(x), \ldots, f^{n_0-1}(x), \ldots \} $ is an $\alpha$-pseudo orbit. Since it is a pseudo-periodic orbit it is $\delta$ shadowed by a periodic point with period $n_0$, call this point $q$. Using that $E^u$ is one dimensional, then
$$\left|\; \frac{1}{n_0}log||Df^{n_0}|E^u_{y_1}|| - \frac{1}{n_0}log||Df^{n_0}|E^u_{y_2}|| \;\right| < \varepsilon.$$

Since we already know that $\lambda^u_f(x)$ exists, this implies that $\lambda^u_f(x)=\lambda_f^u(q)$, hence $\lambda^{u}_f(m) = \lambda^{u}_f(p)$ as we wanted.
\end{proof}

\begin{lemma}
 $$\lambda^{u}_f(m) = \lambda^{u}_A.$$
\end{lemma}
\begin{proof}
 We know that the topological entropy of $A$ is $\lambda_A^{u}$, the conjugacy gives $h_{top}(f)=h_{top}(A)$. From the theory of equilibrium states (\cite{bowen-book}) the measure of maximal entropy is given by the potential $\psi = 0$ and the equilibrium state for the potential $\psi = - log\lambda^{u}$ gives the SRB measure, which is $m$ in our case. And to see that both equilibrium states are the same we just need to see that both potential are cohomologous (\cite{bowen-book}). It means that both measures coincide if, and only if, $$\frac{1}{n}\sum_{i=1}^n(-log ||Df_{f^i(x)}{|E^u}||) = cte, \; \forall x \text{ such that } f^n(x)=x.$$

Which is true by hypothesis.

Finally Pesin's formula gives that $h_f(m)=\int \lambda^{u}_f dm = \lambda_f^{u}$. Let us put all this equalities below.

$$\lambda^{u}_A = h_{top}(A) = h_{top}(f)=h_f(m)=\int \lambda^{u}_f dm=\lambda^{u}_f(p).$$
The lemma is then proved.

\end{proof}
 
 \begin{lemma}
 $$\lambda_f^{ws}(p)=\lambda^{ws}_A$$
\end{lemma}
\begin{proof}

 By the above lemma we already know that $\lambda_f^{u}(p)=\lambda^{u}_A$; and $\lambda_f^{ss}(p) \geq \lambda^{ss}_A$ by Corollary \ref{cor:expoente.estavel.menor}. Hence, since we are on the volume preserving case $\lambda^{ss}_f + \lambda^{ws}_f +\lambda^{u}_f = \lambda^{ss}_A + \lambda^{ws}_A +\lambda^{u}_A$,  therefore we just need to see that $\lambda_f^{ws}(p) \geq \lambda^{ws}_A$ which is the Corollary \ref{cor:lambdaf<lambdaA}. 
\end{proof}

The above lemmas imply,
$$\lambda^*_f(p)=\lambda^*_A(h(p)), \; \forall p \in Per(f).$$

The above equality gives what is known as periodic data, hence by Gogolev, Guysinsky \cite{gogolev.guysinsky-c1.conjugacy} $f$ is $C^1$ conjugate to the linear one.
\hfill $Box$

\subsection{Proof of Theorem \ref{theo.abs.cont.not.c1}}

We start from a linear Anosov with splitting $TM = E^{ss} \oplus E^{ws} \oplus E^{u}$. Let $\phi$ be a volume preserving diffeomorphism which preserves the $E^{ss}$ direction. This means it is absolutely continuous by Gogolev \cite{gogolev-non.abs.cont} and by Lemma \ref{lemma.conjugacy} it is not $C^1$ conjugate as we have changed the exponents.
\hfill $\Box$

\section{Rigidity}\label{sec.rigidity}

The goal of this subsection is to prove Theorem \ref{theo.rigidity}. But first we construct some \textit{Conditional measures with dynamical meaning}. We shall associate to each center leaf a class of measures differing from each other by a multiplication of a positive real number in such a way that on each foliated box the normalized element of this class will give the Rokhlin disintegration of the measure. When the foliation satisfies the hypothesis on Theorem \ref{theo.rigidity} we shall be able to pick measurably on each leaf a representative with some dynamical meaning, it will then help us to obtain some information on the center Lyapunov exponent of $f$.

\begin{lemma}[Avila, Viana, Wilkinson \cite{AVW}]\label{lemma:glue.measures.AVW}
 For any foliation boxes $\mathcal B$, $\mathcal B'$ and m-almost every $x \in \mathcal B \cap \mathcal B'$ the restriction of $m_x^{\mathcal B}$ and $m_x^{\mathcal B'}$ to  $\mathcal B \cap \mathcal B'$ coincide up to a constant factor.
\end{lemma}
\begin{proof}

 Let $\mu _{\mathcal B}$ be the measure on $\Sigma$ obtained as the projection of $m | \mathcal B$ along local leaves. Consider any $\mathcal C \subset \mathcal B$ and let $\mu_{\mathcal C}$ be the projection of $m | \mathcal C$ on $\Sigma$,
\begin{eqnarray*}
 \frac{d \mu_{\mathcal C}}{d \mu_{\mathcal B}} \in (0,1], \; \nu_{\mathcal C} \text{ almost every point}.
\end{eqnarray*}

For any measurable set $E \subset \mathcal C$
$$m(E)= \int_\Sigma m_\xi ^{\mathcal B}(E) \; d\mu_{\mathcal B}(\xi) = \int_\Sigma m_{\xi}^{\mathcal B}(E)\frac{d \mu_{\mathcal B}}{d \mu_{\mathcal C}}(\xi) \; d\mu_{\mathcal C}(\xi).$$

By essential uniqueness, this proves that the disintegration of $m|\mathcal C$ is given by
$$m_\xi ^{\mathcal C} = \frac{d\mu_{\mathcal B}}{d\mu_{\mathcal C}}(\xi) \; m_\xi ^{\mathcal B}; \; \mu_{\mathcal C}(\xi) \text{ almost every point}.$$ 

Take $\mathcal C = \mathcal B \cap \mathcal B'$. Therefore $\frac{d\mu_{\mathcal B}}{d\mu_{\mathcal C}}(\xi) m_\xi ^{\mathcal B} |\mathcal C =  m_\xi ^{\mathcal C} = \frac{d\mu_{\mathcal B'}}{d\mu'_{\mathcal C}}(\xi) m_\xi ^{\mathcal B'} |\mathcal C $. 
Where $\mu'_{\mathcal C}$ is the projection of measure $\mu$ on the transversal $\Sigma '$ relative to the $\mathcal B'$ box. Hence

 $$m_\xi ^{\mathcal B} |\mathcal C = a(\xi) m_\xi ^{\mathcal B'} |\mathcal C,$$ 
where $a(\xi)= \frac{d\mu_{\mathcal B'}}{d\mu'_{\mathcal C}}(\xi) (\frac{d\mu_{\mathcal B}}{d\mu_{\mathcal C}}(\xi))^{-1}.$
\end{proof}

The above lemma implies the existence of a family $\{[m_x] \; | \; x \in M \}$ of measures defined up to scaling 
and satisfying $m_x(M \backslash \mathcal F_x) =0$. The map $x \mapsto [m_x]$ is constant on leaves of $\mathcal F$ and the conditional probabilities $m_x ^{\mathcal B}$
 coincide almost everywhere with the normalized restrictions of $[m_x]$.

We observe that disintegration of a measure is an almost everywhere concept, but in our case, since we shall be considering a $C^1$ center foliation, we look to the conditional measures, of volume, defined everywhere. And, more important, the number $a(\xi)= \frac{d\mu_{\mathcal B'}}{d\mu'_{\mathcal C}}(\xi) (\frac{d\mu_{\mathcal B}}{d\mu_{\mathcal C}}(\xi))^{-1}$ is indeed defined everywhere.

From now on we work on the lift. Let $B:= \mathcal W^{su}(0)$ which is the saturation by unstable leaves of the stable manifold of $0 \in \mathbb R^3$. By the semi-conjugacy we know that every segment of center leaf which has size large enough keep increasing by forward iteration. Let $\gamma_0$ be a length with this property. Let $B_0$ be the two-dimensional topological surface such that each center leaf intersects $B$ and $B_0$ on two points, that are on the same center leaf and at a distance $\gamma_0$ inside the center leaf. Let $B_k := f^k(B_0)$. Therefore, for each point $\xi \in B$ there is a unique point $q_k(\xi)\in B_k$ that is on the same center leaf as $\xi$. Since it will be clear to which point $\xi$ $q_k(\xi)$ is associate, we use $q_k$ instead to simplify notation.

Define the measure $m_{\xi,k}$ by $$m_{\xi,k}([0,q_k])=\lambda^k,$$ where $\lambda$ is the center eigenvalue of the linearization,  $[0,q_k]$ means the segment $[\xi,q_k(\xi)]$ inside the center leaf of $\xi$.
 
\begin{lemma}
 $$f_*m_{x,k} = \lambda^{-1} m_{f(x),k+1}.$$
\end{lemma}
\begin{proof}
 Just see that  $$f_*m_{x,k}([0,q_{k+1}]) = \lambda^{-1} m_{f(x),k+1}([0,q_{k+1}]).$$
\end{proof}

Therefore if the sequence $m_{x,k}$ converges we would get $$f_*m_x = \lambda^{-1} m_{f(x)}.$$

In general, by Lemma \ref{lemma:glue.measures.AVW}, for two foliated boxes $\mathcal B$ and $\mathcal B'$ we have $$m_x^{\mathcal B} \frac{d \nu_{\mathcal B}}{d \nu_{\mathcal C}} = m_x^{\mathcal B'} \frac{d \nu_{\mathcal B'}}{d \nu'_{\mathcal C}} .$$

We apply this formula to the following boxes: $\mathcal B$ and $\mathcal B_k$, where $\mathcal B$ comprehend the segment of center leaves between $B$ and $B_0$, similarly $\mathcal B_k$ is formed by the segment of center leaves bounded by $B$ and $B_k$. Then

$$ m_x^{\mathcal B} . 1 = \frac{d \mu_{\mathcal B _k}}{d \mu_{\mathcal B}} m_x^{\mathcal B _k} = \frac{d \mu_{\mathcal B _k}}{d \mu_{\mathcal B}} \lambda^{-k}m_{x,k}.$$

Note that $\lambda^{k}m_{x,k} = m_x^{\mathcal B _k}$ by the definition of the disintegration. The above proves

\begin{lemma}\label{lemma:m_xk}
On $\mathcal B$:

$$m_{x,k} = (\frac{d \mu_{\mathcal B _k}}{d \mu_{\mathcal B }})^{-1} \lambda^k m_x^{\mathcal B }.$$ 
\end{lemma}

To establish the convergence of the measures we shall need

\begin{lemma} If $\mathcal F^c$ satisfies the hypothesis of Theorem \ref{theo.rigidity} then, there is a uniform constant $\alpha$ such that
 $$\frac{1}{\alpha} \frac{l(\mathcal F^c_x \cap \mathcal B_k)}{{l(\mathcal F^c_x \cap \mathcal B)}} \leq \frac{d \mu_{B_k}}{d \mu_{\mathcal B}}(x) \leq \alpha \frac{l(\mathcal F^c_x \cap \mathcal B_k)}{{l(\mathcal F^c_x \cap \mathcal B)}}.$$
\end{lemma}
\begin{proof}
 To calculate $\frac{l(\mathcal F^c_x \cap \mathcal B_k)}{{l(\mathcal F^c_x \cap \mathcal B)}}$ we need to estimate the volume of a rectangular box. The center holonomy on the center unstable and center stable foliation are bounded by hypothesis. Therefore the volume can be calculated (estimated) by height times base.
\end{proof}

Hence, $$\frac{d \mu_{B_k}}{d \mu_{\mathcal B}}(x) = \alpha_{x,k} \frac{l(\mathcal F^c_x \cap \mathcal B_k)}{{l(\mathcal F^c_x \cap \mathcal B)}},$$ where $\alpha_{x,k} \in [1/\alpha, \alpha]$, for all $x \in \mathbb R^3$ and $k \in \mathbb N$.

Therefore using Lemma \ref{lemma:m_xk} we get on $\mathcal B$
$$m_{x,k} =  \left(\alpha_{x,k} \frac{l(\mathcal F^c_x \cap \mathcal B_k)}{{l(\mathcal F^c_x \cap \mathcal B)}}\right)^{-1} \lambda^k m_x^{\mathcal B}.$$

 For each $x$ there is a subsequence $\alpha_{x,k_{i(x)}}$ that converges to some $\tilde \alpha_x$ as $i(x) \rightarrow \infty$. 

\begin{lemma}
 There is $\beta >0$ such that $\lambda^k / l(\mathcal F^c_x \cap \mathcal B_k) \in [1/\beta, \beta]$ for all $x$.
\end{lemma}
\begin{proof}
 We need to estimate the fraction $$ \frac{||f^n(H(x)) - f^n(H(y))||}{|| A^n(x)-A^n(y) ||} = \frac{||H\circ A^n(x) - H \circ A^n(x)||}{|| A^n(x)-A^n(y) ||}.$$

By the triangular inequality:
\begin{eqnarray*}
\frac{||H\circ A^n(x) - H \circ A^n(y)||}{|| A^n(x)-A^n(y) ||} &\leq& \frac{||H(A^n(x))- A^n(x)||}{|| A^n(x)-A^n(y) ||} + \frac{||A^n(x)- A^n(y)||}{|| A^n(x)-A^n(y) ||} \\ &+& \frac{||H(A^n(y))- A^n(y)||}{|| A^n(x)-A^n(y) ||},
\end{eqnarray*}
and
\begin{eqnarray*}
\frac{||H\circ A^n(x) - H \circ A^n(y)||}{|| A^n(x)-A^n(y) ||} &\geq& -\frac{||H(A^n(x))- A^n(x)||}{|| A^n(x)-A^n(y) ||} + \frac{||A^n(x)- A^n(y)||}{|| A^n(x)-A^n(y) ||} \\ &-& \frac{||H(A^n(y))- A^n(y)||}{|| A^n(x)-A^n(y) ||}.
\end{eqnarray*}
We know that $H$ is at a bounded distance of the identity and $|| A^n(x)-A^n(y) ||$ is big.
\end{proof}

By the above lemma we may assume that $\lambda^k / l(\mathcal F^c_x \cap \mathcal B_k) $ goes to one as $k$ increases, otherwise incorporate it to the constant $\alpha_{x,k}$. Then sending $k_{i(x)}$ to infinity
\begin{eqnarray*}\label{eqn:condi.measure}
m_{x} := \lim_{k_{i(x)} \rightarrow \infty}m_{x,k_{i(x)}} =  (l(\mathcal F^c_x \cap \mathcal B) / \tilde \alpha_x) m_x^{\mathcal B}. 
\end{eqnarray*}

By going to a subsequence we obtained a convergent measure, but we want it to have a specific property. Therefore we have to be more careful on how to define them. We have seen above that  $f_*m_{x,k} = \lambda^{-1} m_{f(x),k+1}$, hence for fixed $x$ there is $k_{i(x)}$ defined as above, but if we define $k_{i(f(x))}=k_{i(x)} +1$ we obtain the limit satisfying $f_*m_{x} = \lambda^{-1} m_{f(x)}$. This means that for fixed $x$ we can define on the orbit of $x$ measures satisfying the mentioned dynamical property.

The measures are in fact indexed on a two dimensional plane manifold $W^{su}$. Hence, to define properly on the whole space, consider the rectangle $A$ such that the intersection of $A$ to the stable manifold of the origin is a fundamental domain. And the sides formed by stable and unstable leaves. Hence defining the measures as we mentioned above on $A$ and on its iterates we get measures with dynamical properties.

From the above we conclude that we did get measures on each center leaf with the property that $f_*m_x = \lambda^{-1} m_{f(x)}$. The construction of such measures  will help us to get information of the center Lyapunov exponent, since we may recover $\lambda$ by the equality
$$\frac{df_*m_x}{dm_{f(x)}}= \lambda^{-1}.$$

Let us explore more deeply the above relation.

\begin{lemma}
By the above notation, the center Lyapunov exponent of $f$ exists everywhere and it is equal to $\lambda$.
\end{lemma}
\begin{proof} Note that
$$\frac{df_*^n m_x}{dm_{f^n(x)}}(f^n(x))= \lambda^{-n}.$$ 
Let us calculate the Radon-Nikodym derivative by another way. Let $I_\delta^n \subset \mathcal F^c_{f^n(x)}$ be a segment of length $\delta$ around $f^n(x)$. Then 
$$ \frac{df_*^n m_x}{dm_{f^n(x)}}(f^n(x))= \lim_{\delta \rightarrow 0} \frac{f_*^n m_x(I^n_\delta)}{m_{f^n(x)}(I^n_\delta)}.$$
And
\begin{eqnarray*}
\frac{df_*^n m_x}{dm_{f^n(x)}}(f^n(x)) 
&=& \lim_{\delta \rightarrow 0} \frac{ m_x(f^{-n}(I^n_\delta))}{m_{f^n(x)}(I_\delta^n) }
=\lim_{\delta \rightarrow 0}\frac{\int_{f^{-n}(I^n_\delta)}  \rho_x d\lambda_x}{\int_{I^n_\delta} \rho_{f^n(x)} d\lambda_{f(x)} } \\
&\thickapprox& \frac{\rho_x(x)}{\rho_{f^n(x)}(f^n(x)) }\lim_{\delta \rightarrow 0}\frac{\int_{f^{-n}(I^n_\delta)}  d\lambda_x}{\int_{I^n_\delta} d\lambda_{f(x)} }
\thickapprox \lim_{\delta \rightarrow 0}\frac{\rho_x(x)}{\rho_{f^n(x)}} \frac{\int_{I^n_\delta} || Df^{-n}||  d\lambda_x}{\int_{I^n_\delta} d\lambda_{f(x)} } \\
&\thickapprox& \frac{\rho_x(x)}{\rho_{f^n(x)}(f^n(x))} || Df^{-n}(x)||.
\end{eqnarray*}

We then have 
$$ \lim_{\delta \rightarrow 0}\frac{df_*^n m_x}{dm_{f^n(x)}}(I^n_\delta)  = \frac{\rho_x(x)}{\rho_{f^n(x)}(f^n(x))} || Df^{-n}(x)||.$$

From the other equalities we have 
$$\frac{\rho_x(x)}{\rho_{f^n(x)}(f^n(x))} || Df^{-n}(x)|| = \lambda^{-n}. $$
By applying "$lim_{n \rightarrow \infty} 1/n \;log$" to the above equality we get
$$\lambda^c(x)=\lambda,$$
since the densities of $m_x$ are uniformly limited.
\end{proof}

We are now ready for the

\textit{Proof of Theorem \ref{theo.rigidity}:}
First, let us prove that $f$ is an Anosov diffeomorphism. We just need to analyze the behavior of $Df$ on the center direction. Let $\varepsilon >0$ be such that $\lambda_\varepsilon := \lambda - \varepsilon > 0$. Since the center exponent exists for every $x$ then, given $x \in \mathbb T^3$, there are $n_x \in \mathbb N$ and a neighborhood $\mathcal U_x$ of $x$ such that $\forall x \in \mathcal U_x$ $|Df^{n_x}|E^c| \geq e^{n_x \lambda_\varepsilon}$. Since $\mathbb T^3$ is a compact manifold take a finite cover $\mathcal U_{x_1} \ldots \mathcal U_{x_l}$. Let $C_i<1$ small enough so that for $x \in \mathcal U_{x_i}$ then $|Df^{n}(x)|E^c| \geq C_{x_i} e^{n \lambda_\varepsilon}$ for all $n \in \{ 0, 1, \ldots, n_{x_i}\}$. Let $C:= \min_i \; C_{x_i}$, we then have that $|Df^{n}(x)|E^c| \geq C e^{n \lambda_\varepsilon}$ for all $x \in \mathbb T^3$ and $n \in \mathbb N$.

 Since, in particular, the center foliation is absolutely continuous, from Gogolev \cite{gogolev-non.abs.cont}, one of the extremal exponents is constant on periodic points. On the other hand the above theorem gives that in particular on the periodic points the central exponent is also constant. Since we are on the conservative case all Lyapunov exponents are constant on periodic points. Then Lemma \ref{lemma.conjugacy} gives that $f$ is $C^1$-conjugate to its linearization.
\hfill $\Box$

\textit{Acknowledgements} This article grew out of my PhD thesis, which was defended at IMPA. I, therefore, would like to thank my former advisor Prof. M. Viana to all the usefull conversations. This work was also greatly influenced from the one month research period I spent at ICMC-USP on September 2011 working with Prof. A. Tahzibi. I leave here all my gratitute to them. This work was partially supported by CNPq and FAPERJ. During the writing of this work the author counted with the support of FAPESP (process \# 2011/21214-3).

\end{document}